\documentclass[reqno]{amsart}
\usepackage{amssymb}
\usepackage{amscd}
\usepackage{mathtools}
\usepackage{comment}
\usepackage{bm}
\numberwithin{equation}{section}

\theoremstyle{plain}

\newtheorem{theorem}{Theorem}[section]
\newtheorem{proposition}[theorem]{Proposition}
\newtheorem{lemma}[theorem]{Lemma}

\newtheorem{problem}[theorem]{Problem}
\newtheorem{conjecture}[theorem]{Conjecture}

\theoremstyle{definition}

\newtheorem{definition}[theorem]{Definition}
\newtheorem{remark}[theorem]{Remark}
\newtheorem{example}[theorem]{Example}

\newcommand{\f}{\mathbb{F}_q}

\newcommand\F{\mathbb{F}}

\newcommand\Z{\mathbb{Z}}

\newcommand\M{\mathcal{M}}

\DeclareMathOperator{\sign}{sign}
\DeclareMathOperator{\type}{type}

\DeclareMathOperator{\mmod}{\ {\rm mod}\ }
\DeclareMathOperator{\lcm}{lcm}

\usepackage[colorlinks=true,linkcolor=blue, citecolor=magenta]{hyperref}

\begin{document}
	
\title[A new sieve for restricted multiset counting]{A new sieve for restricted multiset counting}

\author{Jiyou Li}
\address{Department of Mathematics, Shanghai Jiao Tong University, Shanghai, P.R. China}
\email{lijiyou@sjtu.edu.cn}

\author{Xiang Yu}
\address{Department of Mathematics, City University of Hong Kong, Kowloon Tong, Hong Kong}
\email{xianyu3-c@my.cityu.edu.hk}

\thanks{This work was supported in part by the National Science Foundation of China (11771280, 12031011) and a GRF grant (Project no. CityU 11303718) from the Research Grants Council of the Government of HKSAR, China.}

\begin{abstract}
	The Li--Wan sieve \cite{LD} is extended to multisets when the underlying set is symmetric. The main ingredient of the proof is the Mobius inversion formula on the poset of partitions of $\{1,2,\dots,k\}$ ordered by refinement. As illustrative applications, we investigate the problems of partitions over finite fields and zero-sum multisets over the additive group $\Z/n\Z$.
\end{abstract}
	
\maketitle	

\section{Introduction}

\subsection{Distinct coordinate counting}

 For a positive integer $k$, let  $D^k$ be the Cartesian product of $k$ copies of a set $D$.
 Let $X$ be a subset of $D^k$. Each element $x\in X$ can be written in a vector form $x=(x_1,x_2,\ldots,x_k)$ with $x_i\in D$. Motivated by various problems arising from coding theory and number theory \cite{CM,CW07,CW10,YD}, we are interested in understanding the structure of the set $\overline{X}$ which consists of ``distinct coordinate vectors" in $X$:
\begin{align}\label{def:overlineX}
\overline{X}=\{(x_1,x_2,\ldots,x_k)\in X:x_i\neq x_j,\ \forall\, 1\leq i\neq j\leq k\},
\end{align}
In particular, when $\overline{X}$ is finite, we want to compute its cardinality, or more generally, evaluate complex function sums defined over $\overline{X}$.
%
%

A natural way to compute $|\overline{X}|$ is using the inclusion-exclusion principle. For integers $1\leq i<j\leq k$, let $X_{ij}=\{(x_1,x_2,\ldots,x_k)\in X: x_i=x_j\}$.
Then the classical inclusion-exclusion principle gives
\begin{align}\label{eq:inclusion-exclusion}
|\overline{X}|&=|X\setminus \bigcup_{1\leq i<j\leq k}{X_{ij}}|=|X|-\sum _{1\leq i<j\leq k}|X_{ij}|+\sum_{1\leq i<j\leq k,1\leq
s<t\leq k} |X_{ij}\bigcap X_{st}|-\cdots+(-1)^{k \choose 2} |\bigcap_{1\leq i<j\leq k}X_{ij}|.
 \end{align}
However, the number of terms in the above summation is $2^{{k \choose 2}}$, which easily causes large total errors. In fact, this is a
 major bottle-neck of the inclusion-exclusion sieve.
 In most applications, people use Bonferroni inequalities to get weaker bounds such as
 \begin{align}\label{1.0}
 |\overline{X}|\geq |X|-\sum _{1\leq i<j\leq k}|X_{ij}|.
 \end{align}
These bounds play important roles in many problems in combinatorics, number theory, probability theory and theoretical computer sciences. However, they are usually restrictive. For example,  (\ref{1.0}) is only nontrivial when $|X|>{k \choose 2}$. A natural question is then to find simpler explicit formulas or sharper bounds.

 A formula discovered by Li and Wan \cite{LD} gives an approach to compute $|\overline{X}|$  through a simpler way.
The new formula, which will be described in Theorem \ref{thm:sieve},  shows that there exists a large number of cancellations on the right-hand side of \eqref{eq:inclusion-exclusion}. The number of terms in the summation is significantly reduced from $2^{{k \choose 2}}$ to $k!$, or even fewer, to the partition function $p(k)$ if $X$ is symmetric.

There is a natural action of the symmetric group $S_k$  on elements
of ${X}$ defined as follows. For $\tau\in S_k$ and $x=(x_1,x_2,\ldots,x_k)\in X$, define $\tau \circ x:=(x_{\tau(1)},x_{\tau(2)}, \ldots, x_{\tau(k)})$.
   Let $X_\tau$ be the set of elements in $X$ that are invariant under the action by $\tau$.  Since each $\tau\in S_{k}$ can be written as a product of disjoint cycles $\tau=\tau_1\tau_2\cdots \tau_{c(\tau)}$ uniquely up to the order of the cycles, clearly we have
\begin{equation*}\label{def:Xtau}
X_\tau=\{(x_1,x_2,\ldots,x_k)\in X:x_i=x_j,\ \forall\, i,j \in \tau_m,\ 1\leq m\leq c(\tau)\}.
\end{equation*}

 \begin{theorem}[\cite{LD}, Theorem 1.1]\label{thm:sieve}
	For $\tau\in S_k$, let $\sign(\tau)$  be the  {\it signature} of $\tau$ which is defined by   $\sign(\tau)=(-1)^{k-c(\tau)}$, where $c(\tau)$ is the number of disjoint cycles of $\tau$. Then
	$$ |\overline{X}|=\sum_{\tau\in S_k}\sign(\tau)|X_\tau|.$$
In particular, if $X$ is symmetric, that is, invariant under the action of $S_k$, then
\begin{align}\label{eq:main}
|\overline{X}|=\sum_{\tau\in C_k}\sign(\tau)C(\tau)|X_\tau|,
\end{align}
where  $C(\tau)$ is the size of the conjugacy class of $S_k$ that contains $\tau$ and $|X_\tau|$ is naturally defined over $C_k$, the set of all conjugacy classes of $S_k$.
\end{theorem}

 It is quite surprising that Theorem \ref{thm:sieve} was not known before since the M\"{o}bius inversion over $\Pi_k$ was   known in the 1960s, where $\Pi_k$ is the poset of all partitions of $\{1,2,\dots,k\}$ ordered by refinement.  Precisely, the M\"{o}bius inversion formula gives the following formula for $|\overline{X}|$:
\[   |\overline{X}| =\sum_{\tau\in\Pi_k}\mu(\bm{0}, \tau)|X_\tau|.\]
 In the case of partition  lattice $\Pi_k$,  an explicit expression for
 the M\"{o}bius function $\mu(\bm{0},\tau)$ (see Proposition \ref{prop:Mobius function}) was given independently by
  Sch\"{u}tzenberger in 1954, and Frucht and Rota \cite{Rota} in 1964. However, the above formula is not convenient to use. One  reason is that counting problems over set partitions seem more complicated than those over permutations, as explained below.

   We now explain further why counting over permutations might be
simpler. Suppose that a permutation $\tau\in S_k$ is of type $(c_1,c_2,\ldots,c_k)$, that is,  it has exactly $c_i$ cycles of length $i$. It is well-known that the size of the conjugacy class of $S_k$ that contains $\tau$ is given by
	$$C(\tau)=N(c_1,c_2,\dots,c_k)=\frac{k!}{1^{c_1}c_1!2^{c_2}c_2!\cdots k^{c_k}c_k!},$$
since two permutations in $S_k$ are conjugate if and only if they have the same type. This makes \eqref{eq:main} computable for many interesting cases via  the exponential generating function defined by
\begin{align}\label{eq:exgen}\sum_{k=0}^\infty \sum_{\sum ic_i=k} N(c_1,c_2,\dots,c_k)t_1^{c_1}t_2^{c_2}\cdots t_k^{c_k}\frac{u^k}{k!}=\exp\Big(t_1u+t_2\frac{u^2}{2}+t_3\frac{u^3}{3}+\cdots\Big).\end{align}

The readers are referred to  \cite{LD, li2012} for more details and the proof of Theorem \ref{thm:sieve}. It turns out that the sieve formula \eqref{eq:main} has played an important role in many interesting problems in number theory and coding theory.  

  First, the sieve formula gives an elementary way for enumerating subsets $S$ of $\f^*$ with the property that $\sum_{x\in S}x^m=b$,
 which was first studied by Odlyzko and Stanley for  prime $q$ \cite{OS}. We remark that it has the advantage when used to count the number of $k$-subsets $S$ of $\f^*$ satisfying the same equality \cite{li2013, ZW}.

 Second, since a subset can be naturally regarded as a vector with distinct coordinates,  the sieve formula provides a new counting approach for investigating the subset sum problem, a well-known {\bf{\#P}-complete} problem,  from a mathematical point of view. Precisely,  it is possible to explicitly enumerate subsets of a finite subset $D\subseteq G$ that sum to a given element in $G$, where $G$ is an abelian group.  For example, $G$ could be the additive group of a finite field, the multiplicative group of a finite field, the rational group of an elliptic curve over finite fields, etc.,
and $D$ could be a subset with algebraic structure (a subgroup, for example) or an arbitrarily large subset of $G$.
Many explicit or asymptotic
formulas were obtained for different subsets $D$ in these situations;
see for example \cite{Kos,li2008,li2012,li2018,ZW}. Further applications can be found in \cite{li2012-1, li2019, LWZ}.

In this paper,  we extend the Li--Wan sieve to multisets when $X$ is symmetric. This extension allows us to count  more complicated combinatorial objects naturally, as shown in Section \ref{sec:application}.

\subsection{Motivations for restricted multiset counting}  
We first give the definition of restricted multiset.
\begin{definition}
Let  $D^k$ be the Cartesian product of $k$ copies of a set $D$. A subset $X$ of  $D^k$ is said to be {\it symmetric}  if $(x_{\tau(1)},x_{\tau(2)},\dots,x_{\tau(k)})\in X$ for any $(x_1,x_2,\dots,x_k)\in X$ and any $\tau\in S_k$. From now on, we always assume that $X$ is symmetric. A $k$-multiset $[x_1,x_2,\ldots,x_k]$  is said to satisfy the {\it restriction} $X$ if the ordered $k$-tuple $(x_1,x_2,\ldots,x_k)$ is in $X$. We denote by $\mathcal{M}(X)$ the set of all $k$-multisets satisfying the restriction $X$, that is,
$$\mathcal{M}(X):=\{ [x_1,x_2,\ldots,x_k]: (x_1,x_2,\ldots,x_k)\in X\}.$$
\end{definition}

\begin{example}
Let $D=\mathbb{F}_q$ be a finite field of $q$ elements, and let $X$ be the set $X=\{(x_1,x_2,\dots,x_k)\in D^k:x_1+x_2+\cdots+x_k=0\}$. Then $\mathcal{M}(X)$ consists of $k$-multisets over $\F_q$ whose elements sum to $0$.
\end{example}

\begin{remark}
 It is not hard to check that  the restriction $X$ is well-defined, since $X$ is a symmetric subset of $D^k$. One can also think of the set of  $k$-multisets satisfying the restriction $X$ as the image of $X$ under the map that sends the ordered $k$-tuple $(x_1,x_2,\ldots,x_k)$ to the $k$-multiset $[x_1,x_2,\ldots,x_k]$.
\end{remark}

The problem of counting restricted multisets arises naturally from combinatorics. Some interesting problems are listed as follows.
\subsubsection{Polynomials with prescribed range}
In  studying permutations, hyperplanes and polynomials over finite fields, G\'{a}cs et al. proposed the following conjecture on polynomials with prescribed range.
\begin{conjecture}[\cite{GHNP}, Conjecture 5.1]\label{conj:p}
	Suppose that $M=[a_1,a_2,\dots,a_q]$ is a multiset over a finite field $\F_q$ with $a_1+a_2+\cdots+a_q=0$, where $q=p^h$ with $p$ prime. Let $m<\sqrt{p}$. If there is no polynomial with range $M$ of degree less than $q-m$, then $M$ contains an element of multiplicity at least $q-m$.
\end{conjecture}
Here a multiset $M$ on $\F_q$ is said to be the range of the polynomial $f\in \F_q[x]$ if $M=[f(x):x\in\F_q]$ as a multiset (that is, not only values, but also multiplicities need to be the same).  To prove or disprove the conjecture by a counting argument, a key step is to estimate the cardinality of the set $\mathcal{M}(X)$, where $X$ is the set of ordered $k$-tuples in $(\mathbb{F}_q\setminus\{0\})^k$ that sum to zero. In the case $m=2$, the conjecture holds by Theorem 2.2 in \cite{GHNP}; in the case $m\geq 3$,  the conjecture was disproved by Muratovi\'{c}-Ribi\'{c} and Wang \cite{MW2012} using  an estimation they obtained for $|\mathcal{M}(X)|$.
\subsubsection{Bijection between necklaces and zero-sum multisets}
We may further partition $\mathcal{M}(X)$ into various classes.
  For instance,   consider the set of $k$-multisets in $\mathcal{M}(X)$ with the multiplicity of each element no greater than a given number. For an integer $j\geq 1$,  we define $$ \mathcal{M}_{j}(X)=\{[x_1,x_2,\dots,x_k]\in \mathcal{M}(X):\text{the multiplicity of each}\ x_i \text{\ is no greater than\ } j,\ 1\leq i\leq k\},$$
The set $\mathcal{M}_j(X)$ arises from the bijective proof problem of necklaces and zero-sum multisets \cite{Chan}, which asks for a bijection between cyclic necklaces of length $n$ with at most $q$-colors and zero-sum multisets over $\Z/n\Z$ with the multiplicity of each element strictly less than $q$, when $n$ and $q$ are coprime. Note that the latter is the set $\bigcup_{k=0}^{n(q-1)}\mathcal{M}_{q-1}(X_k)$, where $X_k$ is the set of ordered $k$-tuples in $(\Z/n\Z)^k$ that sum to zero.
Recently, Chan \cite{Chan} gave a surprising bijective construction for this problem  when $q$ is a prime power using tools from finite fields. Specializing to the case $q=2$ answers a question raised by Stanley (see \cite{stanley}, Page 136), which was open for many years. The problem  remains open when $q$ is not a prime power. We believe that our results on restricted multisets might give some new insights into this problem.
%
\subsubsection{List sizes of Reed Solomon codes}

Let  $\f$ be  a finite field of order $q$. Let $1\leq n\leq q$ be a positive integer and $D=\{x_1,x_2,\dots, x_n\} \subset \f$ be a subset of
cardinality $|D|=n>0$. For $1\leq k\leq n$, the Reed-Solomon code
$\mathcal{RS}_{n,k}$ consists of all  vectors of the form  
$$(f(x_1),f(x_2),\dots, f(x_n))\in \f^n,$$
where $f$ runs over all polynomials in $\f[x]$ of degree at
most $k-1$. Reed-Solomon codes play important roles in coding theory.
It is well-known that the minimum distance of the Reed-Solomon code is $n-k+1$. For simplicity we consider the special case $D=\f$ and the corresponding code $\mathcal{RS}_{q,k}$  is called the standard Reed-Solomon code.

  Given a received word  $u$, it is a challenging problem to determine the distance distribution having $u$ as the center. In particular, it
 is an important open problem to obtain list sizes beyond the Jonson bound.
That is, for a non-negative integer $i$, compute the number $N_i(u)$ of codewords in $\mathcal{RS}_{q,k}$ whose distance to $u$ is exactly $i$.
In \cite{li2020}, the authors reduce a list size decoding problem of Reed Solomon codes to a multiset counting problem. For interested readers  we restated it as follows.
\begin{problem} Let $1\leq k \leq q$ and $-k \leq m\leq q-k-1$.
Given a monic polynomial $f(x)\in \f [x]$ of degree $k+m$ and an integer $0\leq r\leq k+m$, count $N(f(x), r)$,  the number  of polynomials $g(x)\in\f[x]$  with  $\deg g(x) \leq k-1$ such that $f(x)+g(x)$ has exactly $r$ distinct roots in $\f$.
\end{problem}

This leads to another refinement of $\mathcal{M}(X)$ and a generalization of the set $\overline{X}$ defined in \eqref{def:overlineX}. It is  the set of $k$-multisets in $\mathcal{M}(X)$ which have exactly $d$ distinct elements, that is,
$$
\overline{\mathcal{M}}_d(X)=\{ [x_1,x_2,\dots,x_k]\in \mathcal{M}(X): [x_1,x_2,\dots,x_k]\ \text{has} \ \text{exactly}\ d\ \text{distinct elements}\}.
$$
For $a\in D$, let $P_a$ be the property that a multiset contains $a$ as an element, and for $A\subseteq D$,  let $N_A$  be the number of $k$-multisets satisfying the restriction $X$ and the property $P_a$ for each $a\in A$.
Then the weighted version of the inclusion-exclusion principle \cite{VW} gives
\begin{align*}
|\overline{\mathcal{M}}_d(X)|&=\sum_{\{a_1, a_2, \dots, a_{d}\}\subset D}N_{\{a_1, a_2, \dots, a_{d}\}}-{d+1 \choose d} \sum_{\{a_1, a_2, \dots, a_{d+1}\}\subset D}N_{\{a_1, a_2, \dots, a_{d+1}\}}+\cdots.
 \end{align*}
However, $N_{\{a_1, a_2, \dots, a_{i}\}}$ is usually depending on
 $\{a_1, a_2, \dots, a_{i}\}$ and thus it seems infeasible to
use this formula to obtain a nice bound on $|\overline{\mathcal{M}}_d(X)|$.


Clearly, the restricted multiset sum problem is a natural generalization of the subset sum problem and thus is \textbf{\#P-complete}. In this paper,
we try to study the counting version of this problem. We establish
several combinatorial identities, which in some cases give interesting closed-form expressions.

\subsection{Main results}

Our idea for computing $|\mathcal{M}(X)|$, $\mathcal{M}_{j}(X)$ and $|\overline{\mathcal{M}}_d(X)|$ is based on the M\"{o}bius inversion formula on $\Pi_k$, the poset of all partitions of $\{1,2,\dots,k\}$ ordered by refinement. The method first appeared in \cite{li2012}.
Given a permutation $\tau\in S_k$,  suppose again that we have a disjoint cycle factorization $\tau=\tau_1\tau_2\cdots \tau_{c(\tau)}$ and the length of the cycle $\tau_i$ is $\ell_i$, $1\leq i\leq c(\tau)$.  For an integer $j\geq 1$ we define
\begin{equation}\label{eq:weightj}
w_{j}(\tau):=(1-(j+1)1_{(j+1)\mid \ell_1})(1-(j+1)1_{(j+1)\mid \ell_2})\cdots (1-(j+1)1_{(j+1)\mid \ell_{c(\tau)}}),
\end{equation}
where $1_{(j+1)\mid\ell_i}$ denotes the indicator function of the statement $(j+1)\mid \ell_i$ which is equal to $1$ if $(j+1)\mid \ell_i$ and $0$ otherwise. Let $d$ be an integer with $1\leq d\leq k$, we define
\begin{equation}\label{eq:weight1}
\overline{w}_d(\tau):=[x^d] (1-(1-x)^{\ell_1})(1-(1-x)^{\ell_2})\cdots (1-(1-x)^{\ell_{c(\tau)}}). 
\end{equation}

Now we can state our main results. Recall that   $X$ is a symmetric subset of  $D^k$ and  $\mathcal{M}(X)$ is defined as $\mathcal{M}(X)=\{ [x_1,x_2,\ldots,x_k]: (x_1,x_2,\ldots,x_k)\in X\}$.
\begin{theorem}\label{thm:j}
	Let $j$ be a positive integer  and let $\mathcal{M}_{j}(X)$ be the set of $k$-multisets in $\mathcal{M}(X)$ with the multiplicity of each element no greater than $j$. Then  we have
	\begin{equation}\label{eq:Mj}
	|\mathcal{M}_{j}(X)| = \frac{1}{k!}\sum_{\tau\in S_k} w_{j}(\tau) |X_\tau|,
	\end{equation}
	In particular, specializing to $j\geq k$, we have
	\begin{equation}\label{eq:MX}
	|\mathcal{M}(X)| = \frac{1}{k!}\sum_{\tau\in S_k}  |X_\tau|.
	\end{equation}
\end{theorem}
\begin{remark}
 The formula \eqref{eq:Mj} can be further simplified by employing the symmetry of $X$. We notice that $X_\tau$ has the same cardinality for $\tau$ in a conjugacy class of $S_k$, since $X$ is symmetric. This leads to the simplification of \eqref{eq:Mj} as
	\begin{align}\label{eq:Mjc}
	|\mathcal{M}_{j}(X)|=\frac{1}{k!}\sum_{\tau\in C_k}w_{j}(\tau)C(\tau)|X_\tau|,
	\end{align}
	where $C_k$ and $C_\tau$ are defined as in Theorem \ref{thm:sieve}. We prefer \eqref{eq:Mj} as it looks cleaner.
	
	 Specializing to $j=1$,  we see from \eqref{eq:weightj} that $$w_1(\tau)=(1-2\cdot 1_{2\mid \ell_1})(1-2\cdot 1_{2\mid \ell_1})\cdots (1-2\cdot 1_{2\mid \ell_{c(\tau)}})=(-1)^{k-c(\tau)},$$
	where we used $1-2\cdot 1_{2\mid \ell_i}=(-1)^{\ell_i-1}$ and $\ell_1+\ell_2+\cdots+\ell_{c(\tau)}=k$. Thus when $j=1$, the sieve formula \eqref{eq:Mj} is indeed the Li--Wan sieve (Theorem \ref{thm:sieve}).
\end{remark}

\begin{theorem}\label{thm:MdX}
Let $d$ be a positive integer and let $\overline{\mathcal{M}}_d(X)$ be the set of $k$-multisets in $\mathcal{M}(X)$ which have exactly $d$ distinct elements. Then
	\begin{equation}\label{eq:MdX}
	|\overline{\mathcal{M}}_d(X)|=\frac{1}{k!}\sum_{\tau\in S_k}\overline{w}_d(\tau)|X_\tau|.
	\end{equation}
\end{theorem}

Theorem \ref{thm:j} and Theorem \ref{thm:MdX} have  natural weighted versions.
 \begin{theorem}\label{thm:weightedj}
 	Let $f:X\to \mathbb{C}$ be a  symmetric function (``symmetric" means $f(x_{\tau(1)},x_{\tau(2)},\ldots,x_{\tau(k)})=f(x_1,x_2,\ldots,x_k)$ for any $(x_1,x_2,\dots,x_k)\in X$ and any $\tau\in S_k$). Then we have
 	\begin{equation*}
 	\sum_{[x_1,x_2,\dots,x_k]\in\mathcal{M}_{j}(X)}f(x_1,x_2,\dots,x_k)= \frac{1}{k!}\sum_{\tau\in S_k} w_{j}(\tau)\sum_{x\in X_\tau}f(x_1,x_2,\dots,x_k).
 	\end{equation*}	
 		In particular, specializing to $j\geq k$, we have
 	\begin{equation*}
 	\sum_{[x_1,x_2,\dots,x_k]\in\mathcal{M}(X)}f(x_1,x_2,\dots,x_k)= \frac{1}{k!}\sum_{\tau\in S_k} \sum_{x\in X_\tau}f(x_1,x_2,\dots,x_k).
 	\end{equation*}
 \end{theorem}

\begin{theorem}\label{thm:weighted}
	Let $f:X\to \mathbb{C}$ be a  symmetric function. Then we have
	\begin{equation*}
	\sum_{[x_1,x_2,\dots,x_k]\in\overline{\mathcal{M}}_d(X)}f(x_1,x_2,\dots,x_k)= \frac{1}{k!}\sum_{\tau\in S_k} \overline{w}_d(\tau)\sum_{x\in X_\tau}f(x_1,x_2,\dots,x_k).
	\end{equation*}	
\end{theorem}

This paper is organized as follows. In Section 2, we prove the sieve formulas, Theorem \ref{thm:j} and Theorem \ref{thm:MdX}, via the M\"{o}bius inversion formula. Then we give two illustrative applications of the sieve formulas in Section 3,

{\bf Notation.} To distinguish between sets and multisets, we use the square bracket notation to denote multisets. Thus for instance, the multiset $\{a,a,b\}$ is denoted by $[a,a,b]$. If $F(x)=\sum_{n=0}^\infty a_nx^n$ is a formal power series, then we use $[x^n]F(x)=a_n$ to denote the coefficient of $x^n$ in $F(x)$. If $S$ is a statement, we use $1_{S}$ to denote the indicator function of $S$, thus $1_{S}=1$ when $S$ is true and $1_S = 0$ when $S$ is false. We often abbreviate partially ordered set as poset.  We use $\bm{0}$ and $\bm{1}$ to denote the least element and the greatest element in a poset, respectively.

\section{Proofs of the main results}

In this section, we prove Theorem \ref{thm:j} and  Theorem \ref{thm:MdX} via the M\"{o}bius inversion formula. 
We first recall the M\"{o}bius inversion formula on posets.

\begin{proposition}[\cite{stanley}, Proposition 3.7.1]\label{prop:Mobius}
	Let $(P,\leq)$ be a poset. Define the M\"{o}bius function $\mu$ of $P$ recursively by
	$$\mu(x,x)=1\ \ \text{for all}\ x\in P, \ \ \mu(x,y)=-\sum_{x\leq z< y}\mu(x,z)\ \ \text{for all}\ x<y\ \text{in}\ P.$$
	Then for $f,g:P\to K$, where $K$ is a field, we have
	$$ g(x)=\sum_{x\leq y}f(y)\  \ \text{for all}\ x\in P$$
	if and only if
	$$ f(x) =\sum_{x\leq y}\mu(x,y) g(y)\  \ \text{for all}\ x\in P.$$
\end{proposition}

Let $\Pi_k$ be the set of all partitions of $\{1,2,\dots,k\}$. Define a partial order $\leq $  on $\Pi_k$ by refinement. That is, declare $\tau\leq \sigma$ if every block of $\tau$ is contained in a block of $\sigma$. Computing the M\"{o}bius function $\mu$ of the poset $(\Pi_k,\leq)$ is a non-trivial result in enumerative combinatorics. We cite it directly from \cite{stanley} without a proof.

\begin{proposition}[\cite{stanley}, Example 3.10.4]\label{prop:Mobius function}
	Let $\tau,\sigma\in \Pi_k$ and $\tau\leq\sigma$. Suppose that $\sigma=\{B_1,B_2,\ldots,B_\ell\}$ and that $B_i$, $1\leq i\leq \ell$ is partitioned into $\lambda_i$ blocks in $\tau$. Then the M\"{o}bius function $\mu(\tau,\sigma)$ is given by
	\begin{equation}\label{eq:mobius}
	\mu(\tau,\sigma)=(-1)^{\lambda_1-1}(\lambda_1-1)!(-1)^{\lambda_2-1}(\lambda_2-1)!\cdots(-1)^{\lambda_\ell-1}(\lambda_\ell-1)!.
	\end{equation}
\end{proposition}

In analogy to the type of a permutation,  a partition $\tau\in \Pi_k$ is said to be of type $(a_1,a_2,\dots,a_k)$ if it has exactly $a_i$ blocks of size $i$, $1\leq i\leq k$. It is not hard to see that the number of partitions in $\Pi_k$ of type $(a_1,a_2,\ldots,a_k)$ is given by
\begin{equation}\label{eq:Na}
\widetilde{N}(a_1,a_2,\ldots,a_k)=\frac{k!}{1!^{a_1}a_1!2!^{a_2}a_2!\cdots k!^{a_k}a_k!}.
\end{equation}

For the purpose of our proof, we need two combinatorial equalities.
\begin{lemma}\label{lem:1}
	Let $\widetilde{N}(a_1,a_2,\dots,a_k)$ be defined as in \eqref{eq:Na} and let $j$ be a positive integer. Then we have
	$$\sum_{\substack{\sum ia_i=k\\a_{j+1}=\cdots=a_{k}=0}} \widetilde{N}(a_1,\ldots,a_k)1!^{a_1}\cdots k!^{a_k}  (-1)^{a_1+\cdots+a_k-1}(a_1+\cdots+a_k-1)!=(k-1)!(1-(j+1)1_{(j+1)\mid k}).
	$$
\end{lemma}
\begin{proof}
		Substituting \eqref{eq:Na} into the above equation, we see that the left-hand side is
{\allowdisplaybreaks
	\begin{align*}
\text{LHS}&=k!\sum_{\substack{\sum ia_i=k}} (-1)^{a_1+\cdots+a_j-1}\frac{(a_1+\cdots+a_j-1)!}{a_1!\cdots a_j!}\\
	&=k!\sum_{\substack{\sum  ia_i=k}} \frac{(-1)^{a_1+\cdots+a_j-1}}{a_1+\cdots+a_j}\binom{a_1+\cdots+a_j}{a_1,\dots,a_j} \\
	&=k!\sum_{m=1}^\infty \frac{(-1)^{m-1}}{m}\sum_{\substack{\sum ia_i=k\\ \sum a_i=m}} \binom{m}{a_1,\dots,a_j}\\
	&=k!\sum_{m=1}^\infty \frac{(-1)^{m-1}}{m}[x^k](x+x^2+\cdots+x^j)^m\\
	&=k![x^k]\log (1+x+\cdots+x^j)\\
	&=k![x^k](\log(1-x^{j+1})-\log(1-x))\\
	&=(k-1)!(1-(j+1)1_{(j+1)\mid k}).
	\end{align*}}
	This proves the lemma.
\end{proof}

\begin{lemma}\label{lem:2}
	Let $\widetilde{N}(a_1,a_2,\dots,a_k)$ be defined as in \eqref{eq:Na} and let $d$ be an integer with $1\leq d\leq k$. Then we have
	$$\sum_{\substack{\sum ia_i=k\\ a_1+\cdots+a_k=d}} \widetilde{N}(a_1,\ldots,a_k)1!^{a_1}\cdots k!^{a_k}  (-1)^{a_1+\cdots+a_k-1}(a_1+\cdots+a_k-1)!=(k-1)!(-1)^{d-1}\binom{k}{d}.
	$$
\end{lemma}
\begin{proof}
	Similar to the proof of the previous lemma, a substitution of \eqref{eq:Na} into the above equation yields
		\begin{align*}
 \text{LHS}& =k!\frac{(-1)^{d-1}}{d}\sum_{\substack{\sum ia_i=k\\ a_1+\cdots+a_k=d}}\frac{d!}{a_1!\cdots a_k!}\\
            & = k! \frac{(-1)^{d-1}}{d} [x^k](x+x^2+\cdots+x^k)^d\\
            & = k! \frac{(-1)^{d-1}}{d} [x^k] (x-x^{k+1})^d(1-x)^{-d}\\
	& =(k-1)!(-1)^{d-1}\binom{k}{d}.
	\end{align*}
	The lemma then follows.
\end{proof}

Now we prove Theorem \ref{thm:j}.
\begin{proof}[Proof of Theorem {\rm\ref{thm:j}}]
	For a partition $\tau\in \Pi_k$, define $X^\circ_\tau$ to be the set of ordered $k$-tuples $(x_1,x_2,\dots,x_k)$ such that $(x_1,x_2,\dots,x_k)\in X_\tau$ but $(x_1,x_2,\dots,x_k)\notin X_\sigma$ for any $\sigma>\tau$.	It is not hard to check that $|X_\tau|=\sum_{\tau\leq \sigma}|X^\circ_\sigma|$. Then the M\"{o}bius inversion formula (Proposition \ref{prop:Mobius})  gives
	\begin{equation}\label{eq:Xcirctau}
	|X^\circ_\tau| =\sum_{\tau\leq \sigma}\mu(\tau,\sigma)|X_\sigma|.
	\end{equation}
	
	Suppose that $\tau=\{B_1,B_2,\dots,B_\ell\}$ and the size of  $B_i$ is $m_i$, $1\leq i\leq \ell$. We observe from the definition of $X_\tau^\circ$ that the multiplicities of elements in the multiset $[x_1,x_2,\dots,x_k]$ with $(x_1,x_2,\dots,x_k)\in X_\tau^\circ$ are $m_1,m_2,\dots,m_\ell$. Thus for $(x_1,x_2,\dots,x_k)\in X_\tau^\circ$,  the multiplicity of each element in $[x_1,x_2,\dots,x_k]$ that is no greater than $j$ is equivalent to the size of each block of $\tau$ that  is no greater than $j$. Since the number of permutations of this multiset is $\binom{k}{m_1,m_2,\dots,m_\ell}$, the number of $k$-multiset satisfying the restriction $X^\circ_\tau$ is
	\begin{equation}\label{eq:MX1}
	|\mathcal{M}(X^\circ_\tau)|=\frac{m_1!m_2!\cdots m_\ell!}{k!}|X^\circ_\tau|.
	\end{equation}
	Note that $X=\bigcup_{\tau\in\Pi_k}X_\tau^\circ$ is a disjoint union of $X_\tau$, so we conclude that
	\begin{align*}
	|\mathcal{M}_{j}(X)|&=\sum_{\tau\in \Pi_k: \text{the size of each block of}\ \tau  \leq j}|\mathcal{M}(X_\tau^\circ)|.
	\end{align*}
    Substituting \eqref{eq:Xcirctau} and \eqref{eq:MX1} into the above equation, we obtain
	\begin{align*}
		|\mathcal{M}_{j}(X)|&=\sum_{\tau\in \Pi_k: m_i\leq j,1\leq i\leq \ell}\frac{m_1!m_2!\cdots m_\ell!}{k!}|X^\circ_\tau|\\
		&=\frac{1}{k!}\sum_{\sigma\in \Pi_k}\Big(\sum_{\tau\leq \sigma:m_i\leq j,1\leq i\leq \ell}m_1!m_2!\cdots m_\ell!\mu(\tau,\sigma)\Big)|X_\sigma|.
	\end{align*}
	Note that here $m_i$ and $\ell$ should be $m_i(\tau)$ and $\ell(\tau)$ respectively, but we omit the variable $\tau$ for notational simplicity.
	
	Since the number of cyclic permutations of length $k$ in $S_k$ is $(k-1)!$, a partition $\sigma$ in $\Pi_k$ with block sizes $n_1,n_2,\dots,n_{r}$ corresponds to $(n_1-1)!(n_2-1)!\cdots (n_{r}-1)!$ permutations in $S_k$.
	Thus to prove \eqref{eq:Mj}, it suffices to show
	\begin{align}\label{eq:Sum}
	\sum_{\tau\leq \sigma:m_i\leq j,1\leq i\leq \ell}m_1!m_2!\cdots m_\ell!\mu(\tau,\sigma)=(n_1-1)!(n_2-1)!\cdots (n_{r}-1)!w_j(\sigma),
	\end{align}
	where $m_1,m_2,\dots,m_{\ell}$ are the block sizes of $\tau$ and $n_1,n_2,\dots,n_{r}$ are the block sizes of $\sigma$.
	We observe from \eqref{eq:mobius} that the sum  on the left-hand side of \eqref{eq:Sum} can be written as a product of the same sum taken over each block of $\sigma$. In view of this and the definition of $w_{j}(\sigma)$,  it suffices to show \eqref{eq:Sum} for partition $\sigma$ with a single block (that is, $\sigma=\bm{1}$), as the general case follows from this special case.
	
	Thus we may assume that  $\sigma=\bm{1}$ and we need to show
\begin{align}\label{eq:sigma1}
\sum_{\tau\in\Pi_k:m_i\leq j,1\leq i\leq \ell}m_1!m_2!\cdots m_\ell!\mu(\tau,\bm{1})=(k-1)!(1-(j+1)1_{(j+1)\mid k}).
\end{align}
Using \eqref{eq:mobius}, the left-hand side can be simplified as
	\begin{align*}
		\sum_{\tau\leq \bm{1}:m_i\leq j,1\leq i\leq \ell}m_1!m_2!\cdots m_\ell!\mu(\tau,\bm{1})&=\sum_{\sum ia_i=k}\sum_{\substack{\tau\in \Pi_k:\type(\tau)=(a_1,\dots,a_k)\\a_{j+1}=\cdots=a_{k}=0}}1!^{a_1}\cdots k!^{a_k}\mu(\tau,\bm{1})\\
		&=\sum_{\substack{\sum ia_i=k\\a_{j+1}=\cdots=a_k=0}}\widetilde{N}(a_1,\dots,a_k)1!^{a_1}\cdots k!^{a_k}(-1)^{a_1+\cdots+a_k-1}(a_1+\cdots+a_k-1)!\\
		&=(k-1)!(1-(j+1)1_{(j+1)\mid k}).
	\end{align*}
	The last step is due to Lemma \ref{lem:1}. This completes the proof.
\end{proof}

Next we prove Theorem \ref{thm:MdX}.
\begin{proof}[Proof of Theorem {\rm \ref{thm:MdX}}]
	A similar argument as in the previous proof yields
	\begin{align*}
	|\overline{\M}_d(X)|=\sum_{\tau\in\Pi_k:\tau\ \text{has eaxctly}\ d\ \text{blocks}}|\M(X_\tau^\circ)|.
	\end{align*}
	Substituting \eqref{eq:Xcirctau} and \eqref{eq:MX1} into the above equation, we obtain
	\begin{align*}
	|\overline{\M}_d(X)|&=\sum_{\tau\in\Pi_i}\frac{m_1!m_2!\cdots m_d!}{k!}|X_\tau^\circ|\\
	&=\frac{1}{k!} \sum_{\sigma\in\Pi_k}\sum_{\tau\leq \sigma}m_1!m_2!\cdots m_d!\mu(\tau,\sigma)|X_\sigma|,
	\end{align*}

   Again, as in the previous proof, the proof will be completed if we can show
   \begin{align*}
    \sum_{\tau\leq \sigma}m_1!m_2!\cdots m_d!\mu(\tau,\sigma)=(n_1-1)!(n_2-1)!\cdots (n_r-1)!\overline{w}_d(\sigma),
   \end{align*}
   where $m_1,m_2,\dots,m_{d}$ are the block sizes of $\tau$ and $n_1,n_2,\dots,n_r$ are the block sizes of $\sigma$, and it can be further reduced to the case that $\sigma=\bm{1}$. Thus we need to show
   \begin{align*}
   \sum_{\tau \in\Pi_k}m_1!m_2!\cdots m_d!\mu(\tau,\bm{1})=(k-1)![x^d](1-(1-x)^k).
   \end{align*}
   Again, using \eqref{eq:mobius}, the left-hand side can be simplified as
   \begin{align*}
   \sum_{\tau \in\Pi_k}m_1!m_2!\cdots m_d!\mu(\tau,\bm{1})&=\sum_{\sum ia_i=k}\sum_{\substack{\tau\in \Pi_k:\type(\tau)=(a_1,\dots,a_k)\\ a_1+\cdots+a_k=d}}1!^{a_1}\cdots k!^{a_k}\mu(\tau,\bm{1})\\
   &=\sum_{\substack{\sum ia_i=k\\ a_1+\cdots+a_k=d}}\widetilde{N}(a_1,\dots,a_k)1!^{a_1}\cdots k!^{a_k}(-1)^{a_1+\cdots+a_k}(a_1+\cdots+a_k-1)!\\
   &=(k-1)!(-1)^{d-1}\binom{k}{d}\\
   &=(k-1)![x^d](1-(1-x)^k).
   \end{align*}
   where we used Lemma \ref{lem:2}. This completes the proof.
\end{proof}
	
The proofs of the weighted versions are omitted since they are completely similar.

\section{Applications to partitions over finite fields and zero-sum multisets over $\Z/n\Z$  }\label{sec:application}

To illustrate the application of our sieve formula, we  investigate two combinatorial problems  which are partitions over finite fields and zero-sum multisets over the group of integers modulo $n$.

\subsection{Partitions over finite fields}
Motivated by the conjecture on polynomials with prescribed range, Muratovi\'{c}-Ribi\'{c} and Wang \cite{MW2013} considered the problem of counting the number of partitions over finite fields. To be precise, let $\F_q$ be a finite field of $q$ elements and $\F_q^*$ be its multiplicative group. A {\it partition} of an element $b\in\F_q$ into $k$ parts is a multiset of $k$ nonzero elements in $\F_q^*$ whose sum is $b$. We denote by $P_k(b)$ the number of partitions of $b$ into $k$ parts over $\F_q$. Using a previous result of Li \cite{li2008} and the inclusion-exclusion principle, Muratovi\'{c}-Ribi\'{c} and Wang obtained an explicit formula for $P_k(b)$. They proved the following theorem:

\begin{theorem}[\cite{MW2013}, Theorem 1]\label{thm:MW}
	Let $k$ be a non-negative integer, $\F_q$ be a finite field of $q=p^a$ elements, and $b\in\F_q$. Define $v(b)=q-1$ if $b=0$ and $v(b)=-1$ otherwise. The number of partitions of $b$ into $k$ parts over $\F_q$ is given by
	$$P_k(b)=\frac{1}{q}\binom{q+k-2}{k}$$
	if $k\not\equiv 0,1\mmod p$,
	$$P_k(b)=\frac{1}{q}\binom{q+k-2}{k}+\frac{v(b)}{q}\binom{q/p+k/p-1}{k/p}$$
	if $k\equiv0\mmod p$, and
	$$
	P_k(b)=\frac{1}{q}\binom{q+k-2}{k}-\frac{v(b)}{q}\binom{q/p+k/p-1}{k/p}$$
	if $k\equiv1\mmod p$.
\end{theorem}

We apply the sieve formula \eqref{eq:MX} to give a direct proof of Theorem \ref{thm:MW}, which avoids using the inclusion-exclusion principle in Muratovi\'{c}-Ribi\'{c} and Wang's proof. First of all, we state a lemma. 

\begin{lemma}[\cite{LD}, Lemma 3.1]\label{lem:LD}
	Assume $p\mid k$. Let $p(k,i)$ be the number of permutations in $S_k$ of $i$ cycles with the length of its each cycle divisible by $p$. Then we have
	$$\sum_{i=1}^k p(k,i)q^i=k!\binom{q/p+k/p-1}{k/p}.$$
\end{lemma}

\begin{proof}[Proof of Theorem {\rm \ref{thm:MW}}]
	Denote by $\widetilde{P}_k(b)$ the number of partitions of $b$ into at most $k$ parts in $\F_q$, that is, the number of multisets of $k$ elements in $\F_q$ whose sum is $b$. It is not hard to see that $P_k(b)=\widetilde{P}_k(b)-\widetilde{P}_{k-1}(b)$. Thus it is sufficient to determine $\widetilde{P}_k(b)$ which, by definition, is the cardinality of the set $\M(X)$ with $X$
	given by
 $$X=\{(x_1,x_2,\dots,x_k)\in \F_q^k:x_1+x_2+\cdots+x_k=b\}.$$
Applying the sieve formula \eqref{eq:MX}, we have
\begin{align}\label{eq:Pk}
\widetilde{P}_k(b)=\frac{1}{k!}\sum_{\tau\in S_k} |X_\tau|.
\end{align}

Suppose that $\tau$ has a disjoint cycle decomposition $\tau=\tau_1\tau_2\cdots \tau_{m}$ and the length of the cycle $\tau_i$ is $\ell_i$, $1\leq i\leq m$. Then we have
$$X_\tau=\{(x_1,x_2,\dots,x_m)\in \F_q^m:\ell_1x_1+\ell_2x_2+\cdots +\ell_mx_m=b\}.$$
If  all the $\ell_i$'s vanish in $\F_q$, that is, $p\mid \ell_i$ for $1\leq i\leq m$, then the above linear equation has $(v(b)+1)q^{m-1}$  solutions and thus $|X_\tau|=(v(b)+1)q^{m-1}$. In particular, in this case, we have $p\mid k$ since $\ell_1+\ell_2+\cdots+\ell_m=k$. Otherwise, the linear equation has $q^{m-1}$ solutions and thus $|X_\tau|=q^{m-1}$. 

When $p\nmid k$, the $\ell_i$'s cannot vanish simultaneously, so we have $|X_\tau|=q^{m-1}$, where $m$ is the number of disjoint cycles of $\tau$. Substituting this into \eqref{eq:Pk}, we conclude that
\begin{align}
\widetilde{P}_k(b)&=\frac{1}{k!}\sum_{i=1}^k c(k,i)q^{i-1}=\frac{1}{q}\binom{q+k-1}{k},
\end{align}
where $c(k,i)$ denotes the unsigned Stirling number of the first kind which counts the number of permutations in $S_k$ with exactly $i$ cycles, and we used the equality $\sum_{i=0}^k c(k,i)x^k=(x+k-1)_k$.

When $p\mid k$, according to the previous discussion, we have
\begin{align*}
\widetilde{P}_k(b)&=\frac{1}{k!}\Big(\sum_{i=1}^k (c(k,i)-p(k,i))q^{i-1}+\sum_{i=1}^k p(k,i)(v(b)+1)q^{i-1}\Big)\\
&=\frac{1}{k!}\Big(\frac{1}{q}\sum_{i=1}^k c(k,i)q^{i}+\frac{v(b)}{q}\sum_{i=1}^k q^i\Big).
\end{align*}
 Using Lemma \ref{lem:LD}, we conclude that
\begin{align*}
\widetilde{P}_k(b)=\frac{1}{q}\binom{q+k-1}{k}+\frac{v(b)}{q}\binom{q/p+k/p-1}{k/p}.
\end{align*}

Finally, noting that $P_k(b)=\widetilde{P}_k(b)-\widetilde{P}_{k-1}(b)$, a discussion depending on whether $k\equiv 0\mmod p$ or $k\equiv 1\mmod p$ completes the proof.

\end{proof}

\subsection{Bijection between necklaces and zero-sum multisets} In his book \cite{stanley}, Stanley raised a bijective proof problem asking for a bijection between  cyclic necklaces with at most two colors and subsets of $\Z/n\Z$ whose elements sum to zero, when $n$ is odd. The problem was answered by Chan \cite{Chan} recently and he generalized the problem to $q$-colored necklaces and multisets which is stated as follows.
\begin{problem}
	Consider these two distinct combinatorial objects: (1) the cyclic necklaces of length $n$ with at most $q$ colors, and (2) the multisets of integers modulo $n$ with elements summing to zero and with the multiplicity of each element being strictly less than $q$. When $q$ and $n$ are coprime, show that these two objects have the same cardinality  and construct a bijection between these two objects.
\end{problem}

In \cite{Chan}, Chan gave a proof of the equinumerosity of these two objects which is not bijective. Additionally, when $q$ is a prime power, he constructed a bijection between these two objects by viewing necklaces as cyclic polynomials over the finite field of size $q$. Note that specializing to $q=2$ answers the bijective proof problem raised by Stanely. Since the bijection that he constructed  relies on finite fields, it fails to work when $q$ is not a prime power. Thus the problem remains open when $q$ is not a prime power.

We would like to use our sieve formula to give another proof of the equinumerosity of these two objects, which is not bijective either. We believe that this proof might give some new insights into this problem. We shall prove the following theorem:

\begin{theorem}\label{thm:NF}
Let $q$ and $n$ be two coprime positive integers.	Let $\mathcal{N}$ denote the set of cyclic necklaces of length $n$ for which the color of each bead is drawn from a color set of size $q$, and let $\mathcal{F}$ denote the set of multisets of elements in $\Z/n\Z$ with element summing to zero and with the multiplicity of each element being strictly less than $q$.  Then we have
	\begin{align}
	|\mathcal{N}|=|\mathcal{F}|=\frac{1}{n}\sum_{e\mid n} \phi(e)q^{n/e},
	\end{align}
	where $\phi$ is Euler's totient function.
\end{theorem}

Before proving the theorem, we state a lemma.

\begin{lemma}\label{lem:j}
Let $N(c_1,c_2\ldots,c_k)$ be the number of permutations in $S_k$ of type $(c_1,c_2,\ldots,c_k)$ and let $t_i=(1-(j+1)1_{(j+1)\mid i})n1_{e\mid i}$. Then we have
$$\sum_{\sum ic_i=k} N(c_1,c_2,\dots,c_k)t_1^{c_1}t_2^{c_2}\cdots t_k^{c_k}=k![u^k](1-u^e)^{-n/e}(1-u^{\lcm(e,j+1)})^{\gcd(e,j+1)n/e}.$$
\end{lemma}
\begin{proof}
Substituting $t_i=(1-(j+1)1_{(j+1)\mid i})n1_{e\mid i}$ into the exponential generating function \eqref{eq:exgen}, we see that the left-hand side of the above equation is
\begin{align*}
\text{LHS}&=k![u^k]\exp\Big((1-(j+1)1_{(j+1)\mid e})n\frac{u^e}{e}+(1-(j+1)1_{(j+1)\mid 2e})n\frac{u^{2e}}{2e}+\cdots\Big)\\
&=k![u^k] \exp\Big(\frac{n}{e}(u^e+\frac{u^{2e}}{2}+\cdots)-\frac{\gcd(e,j+1)n}{e}(u^{\lcm(e,j+1)}+\frac{u^{2\lcm(e,j+1)}}{2}+\cdots)\Big)\\
&=k![u^k]\exp\Big(-\frac{n}{e}\log(1-u^e)+\frac{\gcd(e,j+1)n}{e}\log(1-u^{\lcm(e,j+1)})\Big)\\
&=k![u^k](1-u^e)^{-n/e}(1-u^{\lcm(e,j+1)})^{\gcd(e,j+1)n/e}.
\end{align*}
This completes the proof.
\end{proof}

\begin{proof}[Proof of Theorem {\rm \ref{thm:NF}}]
The cyclic necklaces of length $n$ with at most $q$ color can be viewed as the equivalence class of functions from $\{1,2,\dots,n\}$ to $\{1,2,\dots,q\}$, under the action of the cyclic group $C_n$. Denote the set of the functions by $X$. Then Burnside's lemma gives
$$|\mathcal{N}|=\frac{1}{n}\sum_{g\in C_n} |X_g|,$$
where $X_g$ denotes the set of elements in $X$ that are fixed by $g$. Suppose that $g\in C_n$ is an element of order $e$. Then it is not hard to see that $
|X^g|=q^{n/e}$. Since a cyclic group has $\phi(e)$ elements of order $e$, we conclude that
\begin{align}\label{eq:N}
|\mathcal{N}|=\frac{1}{n}\sum_{e\mid n}\phi(e)q^{n/e}.
\end{align}

Now we consider the cardinality of set $\mathcal{F}$. We note that $\mathcal{F}=\bigcup_{k=0}^{n(q-1)}\M_{q-1}(X_k)$ is a disjoint union of $\M_{q-1}(X_k)$ with
$$
X_k=\{(x_1,x_2,\dots,x_k)\in (\Z/n\Z)^k:x_1+x_2+\cdots+x_k=0\}.
$$
We will use character sums to calculate $|\M_{q-1}(X_k)|$. Let  $G$ denote the additive group $\Z/n\Z$ of order $n$ and let $\widehat{G}$ be the group of all characters on $G$. Using the fact that the sum of all characters over a nonzero element of $G$ is equal to $0$, we have
\begin{align*}
|\M_{q-1}(X_k)|&=\sum_{[x_1,x_2,\dots,x_k]\in \M_{q-1}(G^k)}\frac{1}{|G|} \sum_{\chi\in\widehat{G}}\chi(x_1+x_2+\cdots+x_k)\\
&=\frac{1}{n} \sum_{\chi\in\widehat{G}}\sum_{[x_1,x_2,\dots,x_k]\in \M_{q-1}(G^k)}\chi(x_1)\chi(x_2)\cdots\chi(x_k).
\end{align*}
Applying the sieve formula \eqref{eq:Mj}, we have
\begin{align}\label{eq:Mq1}
|\M_{q-1}(X_k)|=\frac{1}{n}\sum_{\chi\in\widehat{G}}\frac{1}{k!} \sum_{\tau\in S_k} w_j(\tau)\sum_{(x_1,x_2,\dots,x_k)\in G^k_\tau}\chi(x_1)\chi(x_2)\cdots\chi(x_k).
\end{align}

Suppose that $\tau$ has a disjoint cycle decomposition $\tau=\tau_1\tau_2\cdots\tau_m$ and the length of $\tau_i$ is $\ell_i$, $1\leq i\leq m$. We see from the definition of $G^k_\tau$ that
\begin{align}\label{eq:character}
\sum_{(x_1,x_2,\dots,x_k)\in G^k_\tau}\chi(x_1)\chi(x_2)\cdots\chi(x_k)=\prod_{i=1}^m\Big(\sum_{x\in G}\chi^{\ell_i}(x)\Big).
\end{align}
Let $e$ be  the order of the character $\chi$. Then we have  $\sum_{x\in G}\chi^{\ell_i}(x)=|G|=n$ if $e\mid \ell_i$ and $\sum_{x\in G}\chi^{\ell_i}(x)=0$ otherwise. This implies that the sum in the right-hand side of \eqref{eq:character} is equal to $n^m$ if $e\mid \ell_i$ for $1\leq i\leq m$ and $0$ otherwise; in particular, $e\mid k$ in the former case since $k=\ell_1+\ell_2+\cdots+\ell_m$.  Substituting this result into \eqref{eq:Mq1}, we see that
\begin{align*}
&\frac{1}{k!} \sum_{\tau\in S_k} w_j(\tau)\sum_{(x_1,x_2,\dots,x_k)\in G^k_\tau}\chi(x_1)\chi(x_2)\cdots\chi(x_k)\\
=&\frac{1}{k!} \sum_{\sum ic_i=k} \sum_{\tau\in S_k:\type(\tau)=(c_1,c_2,\dots,c_k)} \prod_{i=1}^k (1-q1_{q\mid i})^{c_i}(n1_{e\mid i})^{c_i}\\
=& \frac{1}{k!}\sum_{\sum ic_i=k} N(c_1,c_2,\dots,c_k)\prod_{i=1}^k (1-q1_{q\mid i})^{c_i}(n1_{e\mid i})^{c_i}\\
=& [u^k](1-u^e)^{-n/e}(1-u^{\lcm(e,q)})^{\gcd(e,q)n/e},
\end{align*}
where we used Lemma \ref{lem:j} in the last step. Therefore $|\M_{q-1}(X_k)|$ is simplified as
$$
|\M_{q-1}(X_k)|=\frac{1}{n}\sum_{e\mid n,e\mid k}\phi(e) [u^k](1-u^e)^{-n/e}(1-u^{\lcm(e,q)})^{\gcd(e,q)n/e}.
$$

Summing over $k$, we obtain
\begin{align}\label{eq:MF}
|\mathcal{F}|=\sum_{k=0}^{n(q-1)}|\M_{q-1}(X_k)|=\frac{1}{n}\sum_{e\mid n}\phi(e) \sum_{0\leq k\leq n(q-1):e\mid k}[u^k](1-u^e)^{-n/e}(1-u^{\lcm(e,q)})^{\gcd(e,q)n/e}.
\end{align}
Set $a=\lcm(e,q)/e$ and $b=\gcd(e,q)$. Then we have
\begin{align*}
(1-u^e)^{-n/e}(1-u^{\lcm(e,q)})^{\gcd(e,q)n/e}&=(1-u^e)^{-n/e}(1-u^{ae})^{bn/e}\\
&=\frac{(1-u^{ae})^{n/e}}{(1-u^e)^{n/e}}(1-u^{ae})^{(b-1)n/e}\\
&=(1+u^e+u^{2e}+\cdots+u^{(a-1)e})^{n/e}(1-u^{ae})^{(b-1)n/e}.
\end{align*}
This implies that
$$
\sum_{0\leq k\leq n(q-1):e\mid k}[u^k](1-u^e)^{-n/e}(1-u^{\lcm(e,q)})^{\gcd(e,q)n/e}=q^{n/e}
$$
if $\gcd(e,q)=1$, and
$$\sum_{0\leq k\leq n(q-1):e\mid k}[u^k](1-u^e)^{-n/e}(1-u^{\lcm(e,q)})^{\gcd(e,q)n/e}=0$$
otherwise. Substituting this into \eqref{eq:MF}, we see that
$$|\mathcal{F}|=\frac{1}{n}\sum_{e\mid n:\gcd(e,q)=1} \phi(e)q^{n/e}.
$$
Since $q$ and $n$ are coprime, $|\mathcal{F}|$ can be simply written as
$$
|\mathcal{F}|=\frac{1}{n}\sum_{e \mid n}\phi(e)q^{n/e},
$$
which is exactly the same as \eqref{eq:N}. The proof is completed.
\end{proof}

\end{document}